\documentclass[a4paper,12pt]{article}

\usepackage{amsthm}

\newtheorem{theorem}{Theorem}
\newtheorem{lemma}{Lemma}
\newtheorem{proposition}{Proposition}
\newtheorem{corollary}{Corollary}

\newtheorem{definition}{Definition}
\newtheorem{remark}{Remark}

\newtheorem*{innerthm}{\innerthmname}
\newcommand{\innerthmname}{}
\newenvironment{statement}[1]
 {\renewcommand{\innerthmname}{#1}\innerthm}
 {\endinnerthm}

\usepackage{booktabs}
\usepackage{color}

\usepackage{algorithm}
\usepackage[noend]{algpseudocode}
\makeatletter
\def\BState{\State\hskip-\ALG@thistlm}
\makeatother
\usepackage{algorithmicx}

\usepackage{mathrsfs}

\usepackage{amsmath}
\usepackage{amssymb}
\usepackage{amsfonts}
\usepackage{graphicx}
\usepackage{hyperref}

\usepackage{nicefrac}

\usepackage[normalem]{ulem}
\usepackage[dvipsnames]{xcolor}

\newcommand{\dist}{\mathrm{dist}}
\newcommand{\A}{\mathcal{A}}

\newcommand{\one}{1}

\newcommand{\R}{\mathbb{R}}
\newcommand{\C}{\mathcal{C}}
\newcommand{\N}{\mathbb{N}}

\begin{document}

\title{\bf Converging approximations of attractors via almost Lyapunov functions and semidefinite programming}

\vspace{-15mm}
\date{}
\maketitle

\begin{center}
\author{Corbinian Schlosser$^{1}$}
\end{center}

\footnotetext[1]{CNRS; LAAS; 7 avenue du colonel Roche, F-31400 Toulouse; France. {\tt cschlosser@laas.fr}}
\begin{abstract}
In this paper, we combine two existing approaches for approximating attractors. One of them approximates the attractors arbitrarily well by sublevel sets related to solutions of infinite dimensional linear programming problems. A downside there is that these sets are not necessarily positively invariant. On the contrary, the second method provides supersets of the attractor which are positively invariant. Their method on the other hand has the disadvantage that the underlying optimization problem is not computationally tractable without the use of heuristics -- and incorporating them comes at the price of losing guaranteed convergence. We marry both approaches by combining their techniques and we get converging outer approximations of the attractor consisting of positively invariant sets based on convex optimization via sum-of-squares techniques. The method is easy to use and illustrated by numerical examples.
\end{abstract}

\section{Introduction}

Attractors and Lyapunov functions are intimately connected \cite{teel2000smooth}. Hence it is only natural that Lyapunov functions play a central role in control and asymptotic analysis of dynamical systems. Computing Lyapunov functions is therefore an intriguing but complicated task and often restricted to the case of finding a Lyapunov function for an a-priori given set. That finds extensive application in verification of asymptotic stability of an equilibrium point $x^*$ which lies at the base of many control problems. There are several techniques available to search for Lyapunov functions for a given set $A$ (such as $A = \{x^*\}$), see \cite{giesl2015review}. Among these are sum-of-squares (SOS) techniques that recently gained more and more popularity, beginning with the work of Parillo \cite{parrilo2000structured}, and have further developed since \cite{anderson2015advances}.

In the case of exponentially asymptotically stable fixed points searching for polynomial Lyapunov functions is sufficient \cite{anderson2015advances}. But already dropping the assumption of exponentially fast decay changes this situation, i.e. there exist polynomial dynamical systems with globally asymptotically stable equilibrium point for which no polynomial Lyapunov function exists \cite{ahmadi2011globally}. That shows that the general use of SOS polynomial Lyapunov functions is limited and extensions of the SOS methods are needed in order to treat more general systems. This gets even more drastic when we search for the attractor of the system, that is we do not a-priori know for which set $A$ we want to check if it is an attractor.

Except for Lyapunov based techniques, several other methods for computing/approximating the GA have been proposed. Some approximate the attractor by following trajectories for long but finite time $T \in [0,\infty)$ or using set oriented methods as in \cite{dellnitz2001algorithms}. Another approach motivated by a relaxation of the notion of Lyapunov functions was given in \cite{AlmostLyapunov}. There the authors showed that the weaker concept of Lyapunov function they use is still sufficient to envelope the attractor by positively invariant sets while at the same time, SOS polynomials can be used. The objective function in the underlying optimization problem in \cite{AlmostLyapunov} is hard to evaluate and to make their method computationally tractable the authors relaxed their problem at the cost of guaranteed convergence. Related to this approach is \cite{MilanCorbiAttractor} where the GA is characterized by an infinite dimensional linear programming problem (LP), also solved via a hierarchy of finite dimensional SDPs. Due to the linearity of the cost function in the underlying LP in \cite{MilanCorbiAttractor} the computations are simple and the authors showed convergence towards the attractor but the obtained sets lack the good property of being positively invariant. In this paper we merge both methods into one, maintaining both their advantages -- the approximating sets are positively invariant, guaranteed to convergence, and easy to compute.
\section{Notations}
We denote by $\N$ the natural numbers. The non-negative real numbers are denoted by $\R_+$. We denote the euclidean inner product of $a,b\in \R^n$ by $a \cdot b$. The function $\dist(\cdot,K)$ denotes the Euclidean distance function to a set $K \subset \mathbb{R}^n$ and $\dist(K_1,K_2)$ denotes the Hausdorff distance of two subsets of $\R^n$. For $K \subset \R^n$ we denote by $\overline{K}$ its closure and by $\mathring{K}$ its interior. The space of polynomials (in $n$ variables) is denoted by $\R[x_1,\ldots,x_n]$ or shorter $\R[x]$. The degree of $p \in \R[x]$ is denoted by $\deg(p)$ and the set of polynomials of degree at most $d \in \N$ by $\R[x]_d$. The space of continuous functions on $X$ is denoted by $\C(X)$ and the space of continuously differentiable functions on $\R^n$ by $\C^1(\R^n)$. We denote the gradient of $g \in \C^1(\R^n)$ by $\nabla g$. The pre-image of a set $K$ under a map $g$ is denoted by $g^{-1}(K)$. The Lebesgue measure is denoted by $\lambda$. For $K,K_1,K_2,\ldots \subset \R^n$ we say that $K_m$ converges to $K$ as $m \rightarrow \infty$ with respect to Lebesgue measure discrepancy if $\lim\limits_{m \rightarrow \infty} \lambda(K\setminus K_m) + \lambda (K_m \setminus K) = 0$.

\section{Setting and preliminary definitions}
We consider ordinary differential equations
\begin{equation}\label{EqnODE}
	\dot{x} = f(x), \;\; x(0) = x_0 \in \R^n.
\end{equation}
for a Lipschitz continuous vector field $f:\R^n \rightarrow \R^n$. By $\varphi_{t}(x_0)$ we denote the solution of (\ref{EqnODE}) at time $t\in \R_+$. Further, we consider a compact constraint set $X \subset \R^n$, that is, we are interested only in solutions to (\ref{EqnODE}) for which we have $\varphi_t(x) \in X$ for all $t \in \R_+$. Next, we define (the maximum) positively invariant sets, the basin of attraction, and the two notions of attractors from \cite{AlmostLyapunov}, \cite{MilanCorbiAttractor}, i.e. minimal respectively global attractors. These sets play central roles in this paper.

\begin{definition}\label{def:PosInvMaxInvBasinAttractorLyapunovFunction}
\begin{enumerate}
    \item A set $S$ is called positively invariant for (\ref{EqnODE}) if for all $x \in S$ and $t \in \R_+$ the solution $\varphi_t(x)$ is located in $S$.
    \item The maximum positively invariant set $M_+$ for $X$ is the set
    \begin{equation}\label{def:M+}
	    M_+ = \{x_0 \in X: \varphi_t(x_0) \in X\text{ for all } t \in \R_+\}.
	\end{equation}
	\item A compact set $\A \; \subset M_+ \subset X$ is called a global attractor (GA) for $X$ \cite{Robinson} if it is minimal uniformly attracting, i.e., it is the smallest compact set $\A \subset X$ such that
	\begin{equation}\label{def:Attractor}
	\lim_{t\to\infty} \dist(\varphi_t(M_+),\A) = 0.
	\end{equation}
	\item The basin of attraction of a set $A \subset \R^n$ is the set
	\begin{equation}\label{def:BasinOfAttraction}
        B_f(A):= \{x \in \R^n: \lim\limits_{t\rightarrow\infty}\dist(\varphi_t(x),\A) = 0 \}.
	\end{equation}
	\item A set $A \subset \R^n$ is called asymptotically stable if $B_f(A)$ is an open neighbourhood of $A$ and $A$ is stable, i.e. for all $\varepsilon > 0$ there exists a $\delta > 0$ such that for all $y \in \R^n$ with $\dist(y,A) < \delta$ we have $\dist(\varphi_t(y),A) < \varepsilon$ for all $t \in \R_+$.
	\item A non-empty compact set $A \subset X$ is called stable compact attractor for $X$ if it is positively invariant, asymptotically stable and $B_f(A) \supset X$. We call it minimal attractor (MA) if it is a minimal stable compact attractor for $X$, i.e. no strict subset of $A$ is also a stable compact attractor for $X$.
\end{enumerate}
\end{definition}




\begin{remark}\label{rem:M+InBf} By definition we have $M_+ \subset B_f(\A)$. Note that the GA depends highly on the constraint set $X$ and can be a repelling fixed point, or even the empty set, in cases where $M_+$ is small \cite{MilanCorbiAttractor}.
\end{remark}

Because $X$, and hence also $M_+$ is compact, the GA exists, is unique, positively invariant \cite{Robinson}, and will be denoted by $\A$. Similarly, the MA  is unique as well and will be denoted by $\mathbf{A}$.

\begin{remark}\label{rem:DifferentNotion}
    The method (\ref{PeetMorgan}) from \cite{AlmostLyapunov} treats \textit{minimal attractors} (see \cite{AlmostLyapunov}) and not global attractors in the sense of Definition \ref{def:Attractor}. But both concepts are closely related and coincide under the additional assumption $X \subset B_f(\mathbf{A})$.
\end{remark}

Next, we recall the definition of Lyapunov functions.
\begin{definition}\label{def:LyapunovFunction}
    Let $U \subset \R^n$ be open. A function $V \in \C^1(U)$ is called a Lyapunov function for the dynamical system induced by $f$ if for all $x \in U$
    \begin{equation}\label{def:LyapFunction}
        V(x) \geq 0 \text{ and }\nabla V(x) \cdot f(x) \leq -V(x).
    \end{equation}
    We call $V$ a Lyapunov function for a set $A \subset U$ if $V$ is a Lyapunov function in the above sense and $A := V^{-1}(\{0\})$.
\end{definition}
With regard to Theorem \ref{thm:LyapunovStableReferse} we make the following assumption for the rest of the paper.
\begin{statement}{Assumption 1a}\label{AssumptionAttractorinInterior}
    We assume $B_f(\A) \subset \R^n$ is open, where $\A$ is the GA for $X$.
\end{statement}

\begin{statement}{Assumption 1b}\label{AssumptionAttractorinInterior}
    We assume that the MA $\mathbf{A}$ for $X$ exists..
\end{statement}

Lyapunov functions, MAs, and GAs are intimately related, one reason is the following theorem.

\begin{theorem}[\cite{teel2000smooth}]\label{thm:LyapunovStableReferse}
    A set $A$ is asymptotically stable if and only if there exists a Lyapunov function for $A$. For the MA (under Assumption 1b) that reads we get that $A$ is an MA if and only if there exists a Lyapunov function for $A$.
    Under Assumption 1a that implies that there exists a Lyapunov function $V \in \C^1(B_f(\A))$ for the GA $\A$. Conversely, if $U \subset \R^n$ and $V \in \C^1(U)$ is a Lyapunov function then $V^{-1}(\{0\})$ contains the GA whenever $M_+ \subset U$.
    
\end{theorem}

\section{Two infinite dimensional optimization problems for the global and minimal attractor}
We begin with presenting the infinite dimensional optimization problems for GAs from \cite{MilanCorbiAttractor} and \cite{AlmostLyapunov}. The first one is the LP presented in \cite{MilanCorbiAttractor} which reads
\begin{equation}\label{LPAttractorCorbiMilan}
	\begin{tabular}{llc}
		$p_1^* = \inf$ & $\int\limits_X w(x) \;dx$&\\
		s.t. & $(w,v_1,v_2)\in \C(X) \times \C^1(\R^n) \times \C^1(\R^n)$&\\
			 & $-v^1-v^2 + w \geq \one$& on $X$ \\
			 & $w \geq 0$ & on $X$\\
			 & $ \beta v^1 - \nabla v^1 \cdot f \geq 0$ & on $X$\\
			 & $ \beta v^2 + \nabla v^2 \cdot f \geq 0$ & on $X$
	\end{tabular}
\end{equation}
where $\beta > 0$ is a discounting parameter. To briefly motivate the LP (\ref{LPAttractorCorbiMilan}) the function $w$ should be viewed as an approximation of the indicator function on the GA $\A$ and $v^1$ respectively $v^2$ contain information about the dynamics and can characterize the maximum positively respectively negatively invariant sets $M_+$ and $M_-$ (\cite{MilanCorbiAttractor}) by the points where they are non-negative. The underlying idea for the LP \ref{LPAttractorCorbiMilan} is the representation $\A = M_+ \cap M_-$ \cite{Robinson},\cite{MilanCorbiAttractor}. In \cite{MilanCorbiAttractor} the authors showed $\lambda(\A) = p_1^*$ from (\ref{LPAttractorCorbiMilan}), for the Lebesgue measure $\lambda$. Each feasible $(w,v_1,v_2)$ induces a set $A := w^{-1}([1,\infty))$ that satisfies $\A \subset A \text{ and } \lambda(A \setminus \A) \leq \int\limits_X w \; d\lambda - p_1^*$, and hence, as $(w,v_1,v_2)$ gets optimal, $A$ converges to $\A$ with respect to Lebesgue measure discrepancy \cite{MilanCorbiAttractor}. The method in \cite{AlmostLyapunov} for approximating the MA is based on Lyapunov theory and the optimization problem from \cite{AlmostLyapunov} reads
\begin{equation}\label{PeetMorgan}
	\begin{tabular}{llc}
		$p_2^* = \inf$ & $\lambda\left(J^{-1}([0,1]) \right)$&\\
		s.t. & $J\in \C^1(\R^n)$&\\
		     & $J(x) \geq 0$ for all $x \in X$&\\
			 & $\nabla J \cdot f \leq 1-J$ on $X$& \\
			 & $\emptyset \neq J^{-1}\left([0,1]\right) \subset \mathring{X}$ &
	\end{tabular}
\end{equation}

If the second last constraint would be of the form $\nabla J \cdot f \leq - J$, i.e. $J$ would be a Lyapunov function, it follows directly from Theorem \ref{thm:LyapunovStableReferse} that under Assumption 1b the corresponding optimal value would give $\lambda(\mathbf{A})$. The reason for using $\nabla J \cdot f \leq 1-J$ instead is that there always exist polynomials satisfying the second constraint but this might not be true for the first. In \cite{AlmostLyapunov} the authors showed that still $p_2^*$ from (\ref{PeetMorgan}) equals $\lambda(\mathbf{A})$, the set $J^{-1}([0,1]) \subset X$ is positively invariant, contains the MA and converges to it with respect to Lebesgue measure discrepancy when $J$ gets optimal. 
Compared to the method from \cite{MilanCorbiAttractor} the big advantage is that the sets $J^{-1}([0,1])$ are always positively invariant. On the other hand, the cost term $\lambda\left(J^{-1}([0,1])\right)$ in (\ref{PeetMorgan}) is not linear (or even convex) in $J$ and evaluating the cost is difficult, while that is easy in (\ref{LPAttractorCorbiMilan}). Because there is no known equivalent convex cost for (\ref{PeetMorgan}), the authors in \cite{AlmostLyapunov} use a heuristic as objective function. The resulting SOS program minimizes $\det (P)^{\frac{1}{N_d}}$ for degree bound $d\in \N$, the matrix $P$ representing the SOS polynomial $J$, i.e. $J = z_d(x)^TP z_d(x)$ where $z_d(x)$ is the vector of monomials of degree up to $d$, and $N_d$ the dimension of $\R[x]_d$. With that heuristics, the problem becomes convex -- at the cost of exactness and convergence could not be guaranteed any more. We summarize the differences in the methods from \cite{AlmostLyapunov} (with and without heuristic), \cite{MilanCorbiAttractor} and the method proposed in this text in the following table.
\begin{equation*}
    \begin{array}{ |c|c|c|c|c| } 
 \hline
  & \text{\cite{MilanCorbiAttractor}} & \text{\cite{AlmostLyapunov}} & \text{\cite{AlmostLyapunov} + }& \text{our method}\\ 
  & & & \text{heuristic} & \\
  \hline
 \text{Convex problem} & \checkmark & & \checkmark & \checkmark \\
 \hline
 \text{Invariant sets} &  & \checkmark & \checkmark & \checkmark \\
 \hline
 \text{Convergence} & \checkmark & \checkmark & & \checkmark\\
 \hline
\end{array}
\end{equation*}
The second line, convex problem, refers to the optimization problem being convex, the third line to the property that the obtained sets are positively invariant, and the fourth line to guaranteed convergence of these sets towards the GA.

\section{A combined approach}
Here we will merge the idea of the relaxed Lyapunov equation
\begin{equation}\label{eq:RelaxedLyapunov}
\nabla J \cdot f \leq 1-J
\end{equation}
and the linear structure of (\ref{LPAttractorCorbiMilan}) with an easy evaluation of the cost term. We first note the following result, which is only a reformulation of \cite[Proposition 5]{AlmostLyapunov}, about a perturbed Lyapunov condition.

\begin{lemma}\label{lem:EpsLyap}
    Let $\varepsilon >0$ and $0\leq J \in \C^1(\R^n)$ satisfy
    \begin{equation}\label{eq:LyapPerturbed}
        \nabla J \cdot f \leq \varepsilon -J \text{ on } M_+
    \end{equation}
    then $S:= J^{-1}([0,\varepsilon]) \cap M_+$ is positively invariant and contains the GA. Further, if $S \cap X \subset \mathring{X}$ then $S \cap X \subset M_+$.\\
    For the MA: if (\ref{eq:LyapPerturbed}) holds on $X$ and $\emptyset \neq J^{-1}([0,\varepsilon]) \subset \mathring{X}$ then $J^{-1}([0,\varepsilon])$ contains the MA.
\end{lemma}

\begin{proof} This follows from \cite[Proposition 5]{AlmostLyapunov} by considering the function $\varepsilon^{-1} J$ which satisfies (\ref{eq:RelaxedLyapunov}).
\end{proof}

The close relation between Lyapunov functions and functions solving (\ref{eq:LyapPerturbed}) -- which we refer to by almost Lyapunov functions -- will allow us to transfer Lyapunov function arguments to a linear programming formulation by penalizing $\varepsilon$.

We will do this for both MA and GA and begin with the GA.

\subsection{A combined LP for global attractors}\label{subsec:GA}

We start with an infinite dimensional LP motivated by (\ref{LPAttractorCorbiMilan}) and (\ref{PeetMorgan}) but involving only Lyapunov functions. Therefore we assume Assumption 1a (but we don't assume Assumption 1b).

Later, in the main LP (\ref{LPAttractorLyapunovPerturbedX}) we relax to allow almost Lyapunov functions with the underlying idea of approximating Lyapunov functions for the global attractor from Theorem \ref{thm:LyapunovStableReferse}.  Note that the following infinite dimensional LP considers $M_+$ and not $X$ yet.
\begin{equation}\label{LPAttractorLyapunov}
	\begin{tabular}{llc}
		$p_3^* = \inf$ & $\int\limits_{M_+} w(x) \;dx$&\\
		s.t. & $(w,V)\in \C(M_+) \times \C^1(\R^n)$&\\
			 & $w + V \geq \one$& on $M_+$ \\
			 & $w \geq 0$ & on $M_+$\\
			 & $V \geq 0$ & on $M_+$\\
			 & $ \nabla V \cdot f \leq -V$ & on $M_+$
	\end{tabular}
\end{equation}
\begin{proposition}\label{thm:LyapunovLP}
        We have $p_3^* = \lambda(\A)$ for $p_3^*$ from (\ref{LPAttractorLyapunov}).
\end{proposition}
\begin{proof}
    For any feasible $(w,V)$ the function $V$ is a Lyapunov function for the GA $\A$. By Theorem \ref{thm:LyapunovStableReferse} we have $\A \subset V^{-1}(\{0\})$, i.e. $V = 0$ on $\A$. In particular from $w+V \geq 1$ on $M_+$ it follows $w \geq 1$ on $\A$ and by non-negativity of $w$, we have $\int\limits_{M_+} w(x) \; dx \geq \lambda(\A)$. That means $p^*_3 \geq \lambda(\A)$. To construct a minimizing sequence for (\ref{LPAttractorLyapunov}) let $0\leq V \in \C^1(B_f(\A))$ be a Lyapunov function for $\A$, i.e. $V^{-1}(\{0\}) = \A$, satisfying $\nabla V \cdot f \leq - V$, according to Theorem \ref{thm:LyapunovStableReferse}. By Remark \ref{rem:M+InBf} we have $M_+ \subset B_f(\A)$ and hence for $k \in \N$ the function $w_k := \max\{0, 1-k\cdot V\}$ is continuous on $M_+$ with $w_k + k\cdot V \geq 1$, i.e. the pair $(w_k,k\cdot V)$ is feasible for (\ref{LPAttractorLyapunov}) for all $k \in \N$. For $x \in \A$ we have $V(x) = 0$, thus $w_k(x) = \max\{0,1-k\cdot V(x)\} = \max\{0,1\} = 1$, and for $x \notin \A$ we have $V(x) > 0$, i.e. $w_k(x) = \{0,1-kV(x)\} \searrow 0$ as $k \rightarrow 0$. By the monotone convergence theorem it follows $\int\limits_{M_+} w_k(x) \; dx \rightarrow \lambda(\A)$, hence $p_3^* \leq \lambda(\A)$.
\end{proof}

In the next step, we relax the last constraint in the LP (\ref{LPAttractorLyapunov}) to (\ref{eq:LyapPerturbed}) but we add a penalty for not being a Lyapunov function. Further, we want to work directly on $X$ instead of on the unknown set $M_+$. That will be done through the constraint $\beta v - \nabla v \cdot f\geq 0$. That leads to the final LP with so-called discounting factor $\beta > 0$
\begin{equation}\label{LPAttractorLyapunovPerturbedX}
	\begin{tabular}{llc}
		$p_4^* =$ & $\inf\int\limits_X w(x) \;dx + \varepsilon \lambda(X)$&\\
		s.t. & $(w,J,\varepsilon,v)\in \C(X) \times \C^1(\R^n) \times [0,\infty) \times \C^1(\R^n)$&\\
			 & $w + J - v \geq \one$ \hfill on $X$&\\
			 & $w \geq 0$ \hfill on $X$&\\
			 & $J \geq 0$ \hfill on $X$&\\
			 & $ \nabla J \cdot f + J + v\leq \varepsilon$ \hfill on $X$&\\
			 & $\beta v - \nabla v \cdot f\geq 0$ \hfill on $X$& 
	\end{tabular}
\end{equation}
We will show that also for the LP (\ref{LPAttractorLyapunovPerturbedX}) the optimal value is given by the volume of the GA $\lambda(\A)$.

\begin{theorem}\label{thm:p5=lamba(A)}
    Let $X$ be compact and $f:\R^n\rightarrow \R^n$ be a locally Lipschitz continuous vector field. Let $\A$ be the global attractor for the dynamical system induced by $f$ with constraint set $X$. Then for any $\beta >0$ we have for $p_4^*$ in (\ref{LPAttractorLyapunovPerturbedX})
    \begin{equation*}
        p_4^* = \lambda(\A).
    \end{equation*}
    Further for any feasible $(w,J,\varepsilon,v)$ we have $J^{-1}([0,\varepsilon]) \cap M_+$ is positively invariant and
    \begin{equation}\label{def:K}
    \A \subset K:= J^{-1}([0,\varepsilon]) \cap v^{-1}([0,\infty)) \cap X
    \end{equation}
    with
    \begin{equation}\label{eq:ApproxAttractorLyapunovPerturbEstimate}
        \lambda(K \setminus \A) \leq \int\limits_X w(x) \;dx + \varepsilon \lambda(X) - p^*_5
    \end{equation}
    which converges to zero as $(w,J,\varepsilon,v)$ gets optimal for (\ref{LPAttractorLyapunovPerturbedX}).
\end{theorem}



\begin{proof}
    The essential observation is that any feasible $(w,J,\varepsilon,v)$ satisfies $v \geq 0$ on $M_+$. This is implied by the last constraint in (\ref{LPAttractorLyapunovPerturbedX}) \cite[Lemma 4]{korda2014convex}. It follows that $\nabla J \cdot f \leq \varepsilon - J$ on $M_+$, and hence $J \leq \varepsilon$ on $\A$ and $J^{-1}([0,\varepsilon])$ is positively invariant and contains $\A$ by Lemma \ref{lem:EpsLyap}, as well as we have $w+J \geq 1$ on $M_+$, and hence $w \geq 1-\varepsilon$ on $\A$. That gives
    \begin{align*}
    \int\limits_X w(x) \; dx + \varepsilon \lambda(X) &\geq (1-\varepsilon)\lambda(\A) + \varepsilon\lambda(X) \geq (1-\varepsilon) \lambda(\A) + \varepsilon \lambda(\A) = \lambda(\A),
\end{align*}
    i.e. $p_4^* \geq \lambda(\A)$. The remaining inequality $\lambda(\A) \leq p^*_4$ is the technical part in this proof. We begin by using a construction from \cite{MilanCorbiAttractor} to find a function $v \in \C^1(\R^n)$ with
    \begin{equation}\label{eq:Goodv}
        \beta v - \nabla v \cdot f = 0, \; \; v = 0 \text{ on } M_+ \text{ and } v < 0 \text{ on } X \setminus M_+.
    \end{equation}
    We show that for any $(w,J)$ feasible for (\ref{LPAttractorLyapunov}) and $\varepsilon > 0$ we can find $k = k(\varepsilon)\in \N$ such that $(\tilde{w},\tilde{J},2\varepsilon,k\cdot v)$ is feasible for (\ref{LPAttractorLyapunovPerturbedX}), where $\tilde{w}$ and $\tilde{J}$ are such that the corresponding cost for $(\tilde{w},\tilde{J},2\varepsilon,k\cdot v)$
    is close to the cost of $(w,J)$ for the LP (\ref{LPAttractorLyapunov}). Since $w$ is only non-negative on $M_+$ but (\ref{LPAttractorLyapunovPerturbedX}) requires to be non-negative on $X$ we choose $\tilde{w}$ with $\tilde{w}(x) := \max\{\hat{w} - r \cdot\dist(x,M_+),0\}$ for $r> 0$ large enough (where $\hat{w}$ is any continuous extension of $w$ to $X$, which exists by Tietze's extension theorem) such that
    \begin{equation}\label{eq:CosttildeW}\int\limits_X \tilde{w}(x) \; dx \leq \int\limits_{M+}w(x) \; dx + \varepsilon.
    \end{equation}
    To construct $\tilde{J}$ let $U_1:= J^{-1}([-\nicefrac{\varepsilon}{2},\infty)) \cap X \supset M_+$ and $U_2:= J^{-1}((-\infty,-\varepsilon]) \cap X$. By \cite{lee2013smooth} Theorem 2.29 we can find  a non-negative function $\phi \in \C^1(\R^n)$ with $\phi = 0$ on $U_1 \supset M_+$ and $\phi \geq \min\limits_{x \in X} J(x)$ on $U_2$. Then the function $\tilde{J} := J + \varepsilon +\phi$ is $\C^1$, is non-negative and $\tilde{J}$ (and its derivative) coincides with $J + \varepsilon$ (and its derivative) on $M_+$. Now we consider the choice of $k$ such that $(\tilde{w},\tilde{J}, 2\varepsilon,k\cdot v)$ becomes feasible for (\ref{LPAttractorLyapunovPerturbedX}), i.e. also the first and fourth constraint in (\ref{LPAttractorLyapunovPerturbedX}) are satisfied. Because on $M_+$ we have $w+J + \varepsilon \geq 1 + \varepsilon  >1$, $\nabla J \cdot f + J \leq 0 < \varepsilon$, $\tilde{w} = w$, $\tilde{J} = J$ and $\nabla \tilde{J} = \nabla J$, there is an open neighbourhood $U$ of $M_+$ such that
    \begin{equation}\label{eq:tildew+TildeJ}
        \tilde{w}+\tilde{J} >1 \text{ and } \nabla \tilde{J} \cdot f + \tilde{J} < 2\varepsilon  \text{ on } U
    \end{equation}
    Because $v$ is non-positive and vanishes exactly on $M_+ \subset U_1$ we have $-v \geq \rho$ on $X \setminus U$ for some $\rho > 0$. Let $k \in \N$ with
    \begin{equation}\label{eq:ChoiceK}
        k\geq \rho^{-1}\max\limits_{x \in X \setminus U} \{1 - \tilde{w}(x) - \tilde{J}(x), \nabla \tilde{J}(x) \cdot f(x) + \tilde{J}(x) - 2\varepsilon\}.
    \end{equation}
    By non-positivity of $v$ and (\ref{eq:tildew+TildeJ}) we have
    \begin{equation*}
        \tilde{w}+\tilde{J} - k \cdot v >1 \text{ and } \nabla \tilde{J} \cdot f + \tilde{J} + k \cdot v < 2\varepsilon  \text{ on } U
    \end{equation*}
    For $x \in X \setminus U$ we get by our choice of $k$, (\ref{eq:ChoiceK}), that
    \begin{eqnarray*}
        \tilde{w}(x)+ \tilde{J}(x) -  k\cdot v(x) \geq \tilde{w}(x) + \tilde{J}(x) +  k\rho \overset{(\ref{eq:ChoiceK})}{\geq} 1
    \end{eqnarray*}
    and similarly for the constraint $\nabla \tilde{J} \cdot f + \tilde{J} + k\cdot v \leq 2\varepsilon$. Therefore, $(\tilde{w},\tilde{J},2 \varepsilon,k\cdot v)$ is feasible for (\ref{LPAttractorLyapunovPerturbedX}). Using (\ref{eq:CosttildeW}) we can bound the corresponding cost $\int\limits_X \tilde{w}(x) \;dx + 2 \varepsilon \lambda(X)$ by
    \begin{equation*}
        \int\limits_X \tilde{w} \; dx + 2\varepsilon \lambda(X) \leq \int\limits_{M_+} w(x) \; dx + \varepsilon + 2\varepsilon \lambda(X).
    \end{equation*}
    Since $\varepsilon > 0$ was arbitrary we conclude $p^*_4 \leq p^*_3 = \lambda(\A)$. Finally, it remains to show $\A \subset K$ for $K$ given by (\ref{def:K}) and the estimate (\ref{eq:ApproxAttractorLyapunovPerturbEstimate}) for any feasible $(w,J,\varepsilon,v)$. From Lemma \ref{lem:EpsLyap} and the property $M_+ \subset v^{-1}([0,\infty))$ (see the first line of the proof) it follows $\A \subset K$. Further, by definition of $K$, we obtain from the first constraint in the LP (\ref{LPAttractorLyapunovPerturbedX})
    \begin{equation}\label{eq:estimatew}
        w \geq 1- J + v \geq 1-\varepsilon \text{ on } K.
    \end{equation}
    Non-negativity of $w$ now gives
    \begin{eqnarray}\label{eq:LambdaK}
        \int\limits_X w(x) \; dx + \varepsilon \lambda(X) & \geq & \int\limits_{K} 1-\varepsilon \; dx + \varepsilon \lambda(X) = (1-\varepsilon)\lambda(K) + \varepsilon \lambda(X) \geq \lambda(K).
    \end{eqnarray}
    Subtracting $p^*_5 = \lambda(\A)$ on both sides finishes the proof.
\end{proof}

From Lemma \ref{lem:EpsLyap} we derive the following corollary.

\begin{corollary}\label{cor:InvariantSetInterior}
    In Theorem \ref{thm:p5=lamba(A)}, if $J^{-1}([0,\varepsilon]) \subset \mathring{X}$ then $J^{-1}([0,\varepsilon])$ is positively invariant and the sets $X$ and $v^{-1}([0,\infty)$ can be omitted in (\ref{def:K}).
\end{corollary}

\begin{proof}
    By Lemma \ref{lem:EpsLyap} we have $J^{-1}([0,\varepsilon]) \cap X \subset M_+$ and the crucial estimate (\ref{eq:estimatew}) holds even on $J^{-1}([0,\varepsilon])$.
\end{proof}

Note that \cite{AlmostLyapunov} provides a minimizing sequence for (\ref{LPAttractorLyapunovPerturbedX}) of almost Lyapunov functions $J$ with $J^{-1}([0,\varepsilon]) \subset \mathring{X}$ if $\A \subset \mathring{X}$.

\begin{remark} In light of Remark \ref{rem:DifferentNotion} we mention that our approach can also be applied to the notion of minimal attractors from \cite{AlmostLyapunov}. In contrast to global attractors there is no $M_+$ appearing in the definition of minimal attractors but $X \subset B_f(\A)$ is assumed additionally.  Because the decision variable $v$ in the LP (\ref{LPAttractorLyapunovPerturbedX}) was only needed to incorporate the $M_+$ we can just remove the decision variable $v$ in (\ref{LPAttractorLyapunovPerturbedX}) and obtain an analog result of Theorem \ref{thm:p5=lamba(A)} for minimal attractors based on \cite{AlmostLyapunov} where it was shown $p_2^* = \lambda(\A)$ for $p_2^*$ from (\ref{PeetMorgan}).
\end{remark}

\begin{remark}
    The dual problem of the LP (\ref{LPAttractorLyapunovPerturbedX}) acts on the space of Borel measures on $X$. We did not include the dual problem here because it gives less insight into the GA (\cite{MilanCorbiAttractor}).
\end{remark}

\begin{remark}
Discrete time systems can be handled in an analog way. We refer to \cite{MilanCorbiAttractor} for details about a similar treatment for discrete time systems.
\end{remark}



\subsection{A combined LP for minimal attractors}

In this section, we give an LP formulation for the minimal attractor, again by combining \cite{AlmostLyapunov} and \cite{MilanCorbiAttractor}. In this section, we assume Assumption 1b, i.e. that the MA for $X$ exists.

Due to the definition of MA, we do not have a need for the maximum positively invariant set $M_+$, thus we don't need the decision variable $v$ in (\ref{LPAttractorLyapunovPerturbedX}).

We propose the following LP
\begin{equation}\label{LPMAttractorLyapunovPerturbedX}
	\begin{tabular}{llc}
		$p_{\mathrm{MA}}^* =$ & $\inf\int\limits_X w(x) \;dx + \varepsilon \lambda(X)$&\\
		s.t. & $(w,J,\varepsilon)\in \C(X) \times \C^1(\R^n) \times [0,\infty)$&\\
			 & $w + J \geq \one$ \hfill on $X$&\\
			 & $w \geq 0$ \hfill on $X$&\\
			 & $J \geq 0$ \hfill on $X$&\\
			 & $ \nabla J \cdot f + J \leq \varepsilon$ \hfill on $X$&
	\end{tabular}
\end{equation}
and, similarly to Theorem \ref{thm:p5=lamba(A)}, we show next that $p_{\mathrm{MA}}^* = \lambda(\mathbf{A})$.

\begin{theorem}\label{thm:pMA=lamba(A)}
    Let $X$ be compact and $f:\R^n\rightarrow \R^n$ be a locally Lipschitz continuous vector field. Let $\mathbf{A}$ be the minimal attractor for $X$ for the dynamical system induced by $f$. Then for $p_{\mathrm{MA}}^*$ in (\ref{LPAttractorLyapunovPerturbedX}) it holds
    \begin{equation*}
        p_{\mathrm{MA}}^* = \lambda(\mathbf{A}).
    \end{equation*}
    For any feasible $(w,J,\varepsilon)$ we have
    \begin{equation}\label{def:MAK}
    \mathbf{A} \subset K:= J^{-1}([0,\varepsilon]) \cap X
    \end{equation}
    with
    \begin{equation}\label{eq:ApproxMAttractorLyapunovPerturbEstimate}
        \lambda(K \setminus \mathbf{A}) \leq \int\limits_X w(x) \;dx + \varepsilon \lambda(X) - p^*_{\mathrm{MA}}
    \end{equation}
    which converges to zero as $(w,J,\varepsilon,v)$ gets optimal for (\ref{LPAttractorLyapunovPerturbedX}). Further $J^{-1}([0,\varepsilon])$ is positively invariant if $J^{-1}([0,\varepsilon]) \subset \mathring{X}$.
\end{theorem}

\begin{proof}
    We argue similarly as in Section \ref{subsec:GA} but can avoid many technical steps because the set $M_+$ (and correspondingly $v$) does not appear in the LP (\ref{LPMAttractorLyapunovPerturbedX}). Using a Lyapunov function $V$ for $\mathbf{A}$ shows, as in Proposition \ref{thm:LyapunovLP} that $p_{\mathrm{MA}}^* \leq \lambda(\mathbf{A})$. The inequality $p_{\mathrm{MA}}^* \geq \lambda(\mathbf{A})$ follows as in the proof of Theorem \ref{thm:p5=lamba(A)}. For any feasible $(w,J,\varepsilon)$ we have $J\leq \varepsilon$ on $\mathbf{A}$. Hence $w+J \geq 1$ implies $w \geq 1-\varepsilon$ on $\mathbf{A}$. That gives
    \begin{align*}
    \int\limits_X w(x) \; dx + \varepsilon \lambda(X) &\geq (1-\varepsilon)\lambda(\mathbf{A}) + \varepsilon\lambda(X) \geq (1-\varepsilon) \lambda(\mathbf{A}) + \varepsilon \lambda(\mathbf{A}) = \lambda(\mathbf{A}),
\end{align*}
    That $J^{-1}([0,\varepsilon]) \cap X$ contains $\mathbf{A}$ as well as the last statement in Theorem \ref{thm:pMA=lamba(A)} follows from Lemma \ref{lem:EpsLyap}. The estimate (\ref{eq:ApproxMAttractorLyapunovPerturbEstimate}) follows as for (\ref{eq:ApproxAttractorLyapunovPerturbEstimate}) via (\ref{eq:LambdaK}) from $J\leq \varepsilon$ on $\mathbf{A}$.
\end{proof}

\section{Solving the linear programs}\label{Sec:SolvingTheLinearPrograms}

As in \cite{AlmostLyapunov} and \cite{MilanCorbiAttractor} we approach the infinite dimensional LPs (\ref{LPAttractorLyapunovPerturbedX}), (\ref{LPMAttractorLyapunovPerturbedX}) via polynomials and sum-of-squares techniques based on polynomial optimization \cite{Las01}, \cite{lasserre2009moments}. The resulting SOS problem can be reformulated as a semidefinite program (SDP) \cite{lasserre2009moments} \cite{Las01} and the LP (\ref{LPAttractorLyapunovPerturbedX}) will be solved via a hierarchy of these finite dimensional SDPs. We only present the corresponding SOS respectively SDP problems for the LP (\ref{LPAttractorLyapunovPerturbedX}) for the GA. For the MA the procedure is similar.

The idea is to replace the decision variables $w,J,v$ by polynomials (this will be justified by the Stone-Weierstraß theorem) and the non-negativity will be algebraically certified by an SOS condition. This is a standard procedure and we refer to~\cite{lasserre2009moments} and \cite{korda2014convex} for details. In order to apply these algebraic certificates we treat polynomial vector fields $f$ and compact basic semi-algebraic constraint sets $X$. A set $X$ is called compact basic semi-algebraic if there exist polynomials $g_1,\ldots,g_j \in \R[x_1,\ldots,x_n]$ such that
\begin{equation*}
    X = \{x \in \R^n: g_i(x) \geq 0 \text{ for } i = 1,\ldots,j\}.
\end{equation*}
Further, one of the $g_i$ is given by $g_i(x) = R_X^2 - \|x\|_2^2$ for some $R_X \in \R$. For each degree $k \in \N$ we get the following SOS program
\begin{equation}\label{SDP}
	\begin{tabular}{llc}
		$d_k:=$ & $\inf \mathbf{w}' \mathbf{l} +  \varepsilon \mathbf{l}_0$& \vspace{0.5mm}\\
		s.t. & $w,J,v \in \R[x]_k, \varepsilon \in \R$, $p_0,q_0,r_0,s_0 \in \R[x]_{\nicefrac{k}{2}}$&\\
		     & $p_i,q_i,r_i,s_i \in \R[x]_{\nicefrac{(k-\deg (g_i))}{2}}$ for $i = 1,\ldots,j$&\\
		     & $\varepsilon \geq 0$&\\
			 & $w + J - v - 1 = p_0^2 + \sum\limits_{i = 1}^j p_i^2 g_i$&\\
			 & $w(x) = q_0^2 + \sum\limits_{i = 1}^j q_i^2 g_i$&\\
			 & $ \varepsilon - \nabla J \cdot f - J - v = t_0^2 + \sum\limits_{i = 1}^j t_i^2 g_i$ &\\
			 & $ \beta v - \nabla v \cdot f = s_0^2 + \sum\limits_{i = 1}^j s_i^2 g_i$ &
	\end{tabular}
\end{equation}
where $\mathbf{w}'$ is the vector of coefficients of the polynomial $w$ and $\mathbf{l}$ is the vector of the moments of the Lebesgue measure on $X$ (i.e., $\mathbf{l}_\alpha = \int_X x^\alpha \, dx$, $\alpha \in \mathbb{N}^n$, $\sum_i \alpha_i \le k$), both indexed in the same basis of $\R[x]_k$, which gives $\mathbf{w}'\mathbf{l} = \int\limits_X w(x) \; dx$ and $\mathbf{l}_0$ corresponds to $\alpha = 0$, i.e. $\mathbf{l}_0 = \lambda(X)$.

The following theorem shows that the properties of the LP (\ref{LPAttractorLyapunovPerturbedX}) are inherited to the finite dimensional SOS program and that their optimal values converge to $\lambda(\A)$ as well.

\begin{theorem}
    For (\ref{SDP}) we have $d_k \searrow \lambda(\A)$ as $k \rightarrow \infty$, for any feasible $(w_k,J_k,\varepsilon_k,v_k)$ (and corresponding $p_i,q_i,r_i,s_i$) the set $J_k^{-1}([0,\varepsilon_k])$ is positively invariant and the set $A_k:= J_k^{-1}([0,\varepsilon_k]) \cap v_k^{-1}([0,\infty) \cap X$ contains the GA $\A$. If $(w_k,J_k,\varepsilon_k,v_k)$ is optimal for (\ref{SDP}) then we have\vspace{-1.7mm}
    \begin{equation}\label{eq:convergenceEstimateSDP}
        \lambda(A_k \setminus \A) \leq d_k - \lambda(\A) \rightarrow 0 \text{ as } k \rightarrow \infty.
    \end{equation}
    $ $
\end{theorem}
\vspace{-4mm}
\begin{proof} Since the SOS problems (\ref{SDP}) are tightenings of the LP (\ref{LPAttractorLyapunovPerturbedX}) we only need to show convergence and (\ref{eq:convergenceEstimateSDP}). The proof is very similar to the one of \cite[Theorem 5 and 6]{MilanCorbiAttractor}. 
    The additional perturbation parameter $\varepsilon$ guarantees that we can always find (almost) optimal polynomials $J$ satisfying the constraint $\nabla J \cdot f + J \leq \varepsilon$ by increasing $\varepsilon$ slightly if necessary. To see this let $(w,J,\varepsilon,v)$ be feasible for the LP (\ref{LPAttractorLyapunovPerturbedX}). Then for any $\delta > 0$ the quadrupel $(w+\delta,J+ \delta,\varepsilon + 3\delta, v + \delta)$ is strictly feasible and the cost only changes by $4\delta \lambda(X)$. By the Stone-Weierstraß theorem we can find polynomials $(p_w,p_J,p_v)$ close enough to $(w+\delta,J+\delta ,v+\delta)$ in the topology of $\C^1$ such that $(p_w,p_J,\varepsilon + 3\delta,p_v)$ is still strictly feasible for the LP (\ref{LPAttractorLyapunovPerturbedX}) (because $(w+\delta,J+ \delta ,\varepsilon + 3\delta, v + \delta)$ is strictly feasible). An SOS representation of $(p_w,p_J,p_v)$ follows then from Putinar's positivstellensatz \cite{Putinar} which shows that $(p_w,p_J,p_v)$ is feasible for (\ref{SDP}) for large enough $k$. Since $\delta > 0$ was arbitrary it follows that $d_k \rightarrow p^*_5 = \lambda(\A)$. The proof for (\ref{eq:convergenceEstimateSDP}) is the same as the one for (\ref{eq:ApproxAttractorLyapunovPerturbEstimate}).
\end{proof}

A similar extension as in Corollary \ref{cor:InvariantSetInterior} can be obtained by the same arguments.

The cost of having guaranteed bounds and finding the global optima using SOS methods comes at the price of expensive scaling, that is, high degree $d$ or dimension $n$ makes the corresponding SDP intractable for current solvers/memory/computation power. Therefore, further structure should be exploited, such as sparsity \cite{schlosser2020sparse}, \cite{tacchi2020approximating}, \cite{wang2021exploiting} or symmetry \cite{goluskin2019bounds}. 

\section{Numerical examples}
We illustrate our approach by three numerical examples that have been used in \cite{MilanCorbiAttractor} and \cite{AlmostLyapunov}. The first one is the following globally asymptotically stable system with attractor $\A = \{(0,0)\}$, which does not allow for a polynomial Lyapunov function\cite{ahmadi2011globally}
\begin{align}\label{eq:NoLyapunov}
    \dot{x}(t) &= -2y(t)\left(-x(t)^4 +2x(t)^2y(t)^2+y(t)^4\right) -\notag\\
    & \; \; \; 2x(t)(x(t)^2+y(t)^2)\left(x(t)^4+2x(t)^2y(t)^2-y(t)^2\right)\notag\\
    \dot{y}(t) &= 2x(t)\left(x(t)^4 + 2x(t)^2y(t)^2-y(t)^4\right) - \notag\\
    & \; \; \;  2y(t)(x(t)^2+y(t)^2)\left(-x(t)^4+2x(t)^2y(t)^2+y(t)^4\right).
\end{align}
The second example is the Van--der--Pol oscillator
\begin{equation}\label{eq:VanderPol}
	\dot{x}(t) = 2y(t),  \quad \dot{y}(t) = -0.8x(t) - 10(x(t)^2-0.21)y(t).
\end{equation}
And the third example is the H\'{e}non map as an example of discrete time systems, given by
\begin{equation}\label{eq:Henon}
	x_{m+1} = \frac{2}{3}(1+y_m) - 2.1 x_m^2, \quad  y_{m+1} = 0.45 x_m.
\end{equation}

For the Van--der--Pol oscillator we observe as in \cite{AlmostLyapunov} and \cite{MilanCorbiAttractor} that the proposed method works very well and is comparable with the method in\cite{AlmostLyapunov}, see Figure \ref{figVanDerPol}. We notice that the approximation from \cite{MilanCorbiAttractor} seems to perform slightly better than our approach in terms of volume discrepancy with the real attractor.

\begin{figure}[!h]
\begin{picture}(0,200)
\put(-20,0){\includegraphics[width=95mm]{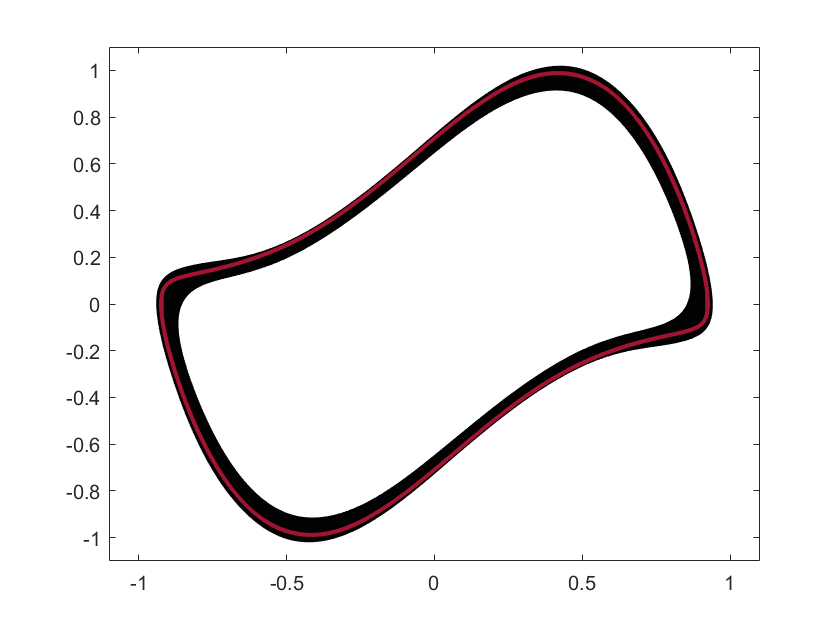}}
\put(240,10){\includegraphics[width=80.5mm]{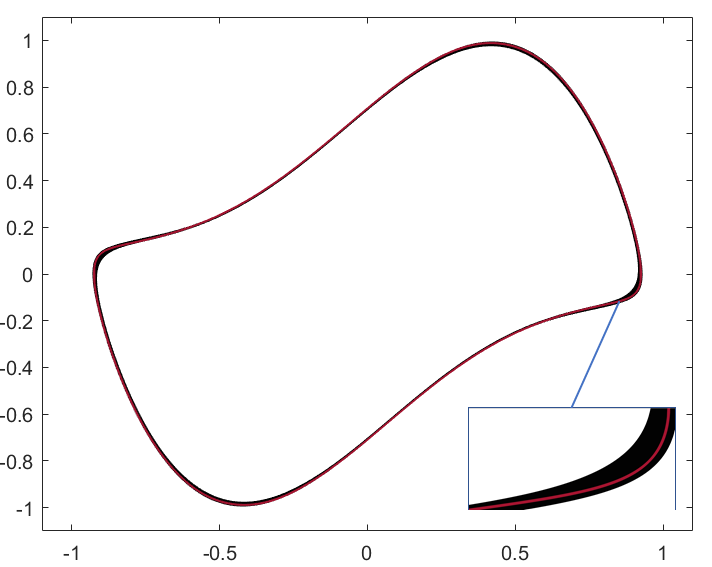}}
\end{picture}
\vspace{-1mm}
\caption{\footnotesize{Outer approximations (black) of the attractor (red) for the Van-der-Pol oscillator for $X = \{x : 0.4 \le\|x\|_2 \le 2\}$. Left: approximation, degree 12 polynomials, and $\beta = 0.2$. Right: approximation for polynomials up to degree 16 and $\beta = 0.2$.}}\label{figVanDerPol}
\end{figure}

For the system (\ref{eq:NoLyapunov}) we notice some numerical instabilities in the decision variable $\varepsilon$ in (\ref{SDP}) when solving the SDPs using Yalmip \cite{lfberg2004toolbox} and Mosek \cite{mosek} (Figure \ref{figNoLyapunov} left). Using bisection in $\varepsilon \geq 0$ (for small $\varepsilon$) and solving the corresponding SDPs (\ref{SDP}) for fixed $\varepsilon$ avoided the mentioned numerical issues.
\begin{figure}[!h]
\begin{picture}(0,200)
\put(-20,-0.7){\includegraphics[width=89.5mm]{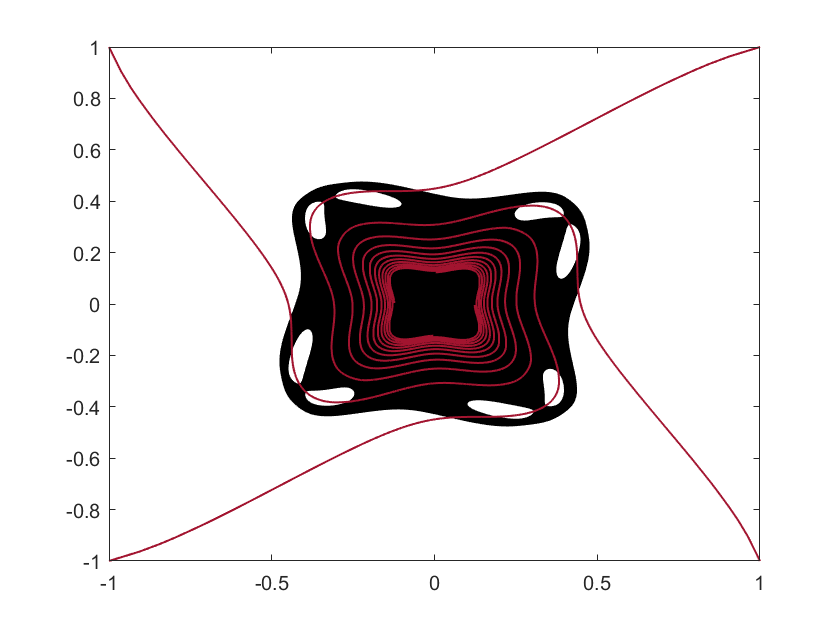}}
\put(220,-1.8){\includegraphics[width=90mm]{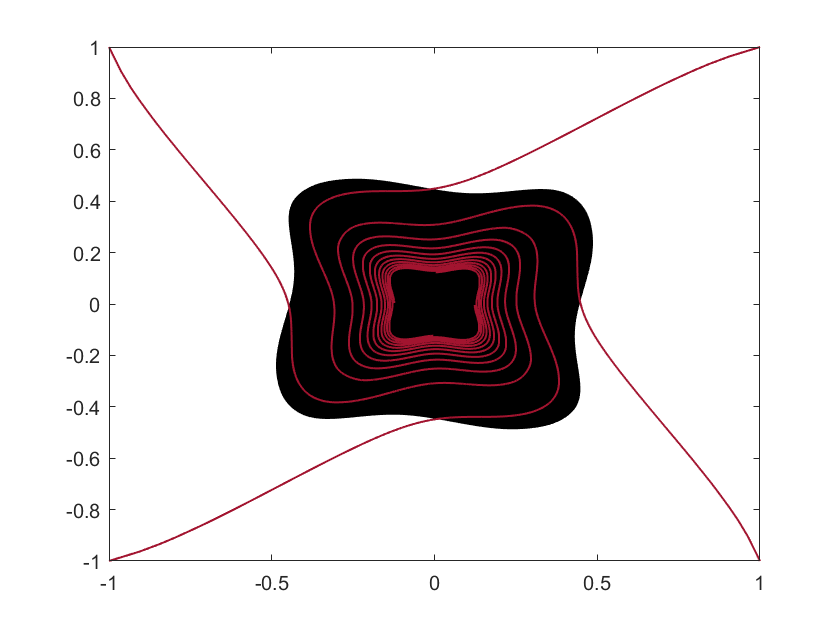}}
\end{picture}
\vspace{-1mm}
\caption{\footnotesize{Outer approximations (black) of the attractor $\A = \{(0,0)\}$ and trajectories starting from $(1,1),(1,-1),(-1,1),(-1,-1)$ (red) for (\ref{eq:NoLyapunov}) for $X =[-1,1]^2$. Left: approximation by degree 16 polynomials and $\beta = 0.2$, the obtained $\varepsilon^*$ in (\ref{SDP}) is too small and causes incorrect behaviour of the set $J^{-1}([0,\varepsilon^*])$, see white ``holes". Right: Outer approximation using bisection on $\varepsilon$ and polynomials up to degree 16 with discounting parameter $\beta = 0.2$.}}\label{figNoLyapunov}
\vspace{-2mm}
\end{figure}


In comparison with \cite{MilanCorbiAttractor} the situation for the H\'{e}non map (Figure \ref{figHenon} left) is similar to the one for the Van-der-Pol oscillator. It takes higher degree polynomials to capture the complex topology of the H\'{e}non attractor compared to \cite{MilanCorbiAttractor}.

\begin{remark}
    The discounting parameter $\beta > 0$ can be tuned and several solutions corresponding to different values of $\beta$ can be intersected to improve the quality of the approximation \cite{MilanCorbiAttractor}. Similarly, we can introduce a parameter $\gamma > 0$ to the ``almost Lyapunov" constraint by considering $\nabla J \cdot f \leq \varepsilon-\gamma \cdot J$. As for $\beta$, small values of $\gamma$ describe less/slower discounting/decay and should be used when the dynamics towards the attractor are slow. The intersection of solutions for different values of $\beta$ and $\gamma$ for the H\'{e}non map is illustrated on the right in Figure \ref{figHenon}.
\end{remark}

\begin{figure}[!h]
\begin{picture}(0,200)
\put(-20,0){\includegraphics[width=93mm]{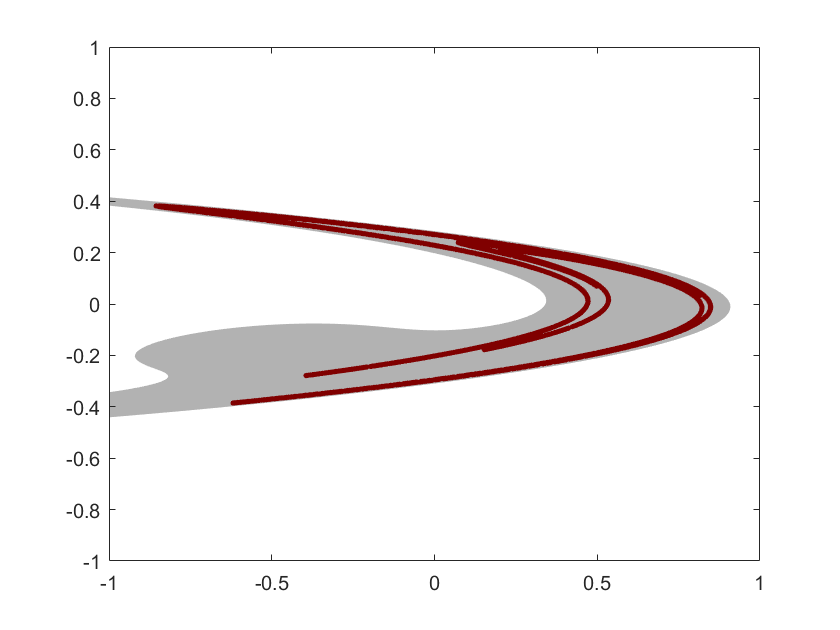}}
\put(220,0){\includegraphics[width=93mm]{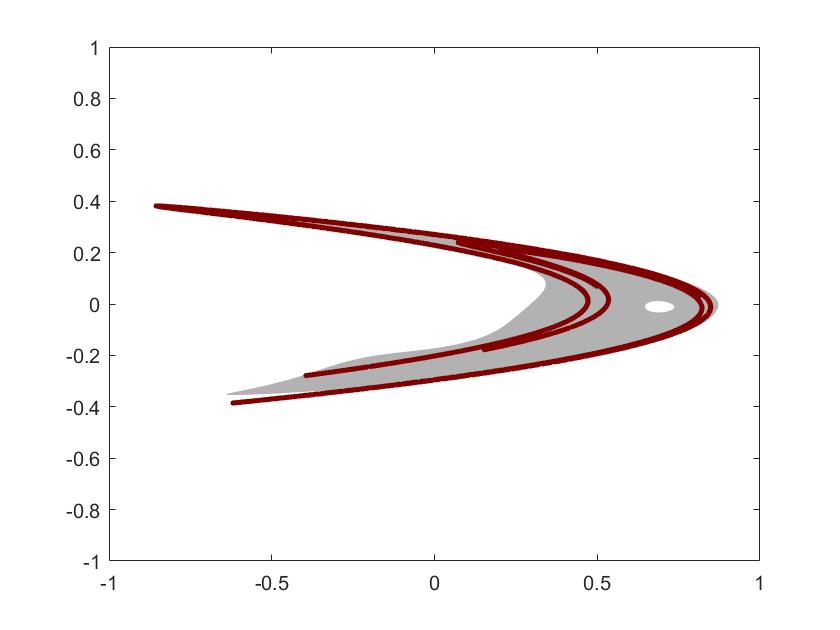}}
\end{picture}
\vspace{-1mm}
\caption{\footnotesize{Outer approximations (gray) of the attractor (red) for the H\'{e}non map for $X = [-1,1]^2$. Left: approximation, degree 6 polynomials and $\alpha = 0.002$, $\gamma = 0.05$. Right: Intersection of approximation by degree 8 polynomials obtained by different values $\beta = 0.001, 0.002, 0.01$ and $\gamma = 0.002, 0.05, 0.2$.}}\label{figHenon}
\end{figure}

\section{Conclusion}
We presented a linear programming approach to outer approximations of global attractors via positively invariant sets. This builds on the recent works in \cite{AlmostLyapunov} and \cite{MilanCorbiAttractor}. We combine both methods by keeping their fundamental advantages. That is: We use the approximation via positively invariant sets from \cite{AlmostLyapunov} by using their method of perturbed Lyapunov equation and we overcome their difficulty in evaluating the cost function by maintaining the linear structure of the optimization problem from \cite{MilanCorbiAttractor}.

This leads to an infinite dimensional linear programming problem characterizing the GA (up to Lebesgue measure discrepancy zero) via certain pre-images of functions that are feasible for the optimization problem. Applying sum-of-squares techniques as in \cite{AlmostLyapunov} and \cite{MilanCorbiAttractor} allows us to formulate a converging hierarchy of semidefinite programs. This gives rise to convergent outer approximations of the GA by positively invariant semialgebraic sets that are easy to compute. We illustrate the approach with numerical examples, including one of a vector field that does not allow for a polynomial Lyapunov function.

With regard to applications, we think that our approach should be understood as a practical extension of \cite{AlmostLyapunov} and as a qualitative extension of the previous work in \cite{MilanCorbiAttractor}, where it seems that the GA is approximated with less Lebesgue measure discrepancy but not necessarily by positively invariant sets.

Apart from approximating the GA, this work can be extended to bounding extreme events on attractors based on the work \cite{goluskin2020attractor}.

Another possible direction could be towards data based algorithms for attractors as in \cite{korda2020computing} or towards attractors of partial differential equations -- where Lyapunov functions also provide a powerful tool -- via for instance \cite{korda2018moments}, \cite{chernyavsky2021convex}.

\section{Acknowledgements}
The author is very thankful to the anonymous reviewers for their insightful comments, remarks, and corrections on previous versions of the paper, including very helpful suggestions in the context of readability.

This work has been supported by the European Union's Horizon 2020 research and innovation programme under the Marie Sk\l{}odowska-Curie Actions, grant agreement 813211 (POEMA).







\bibliographystyle{plain}
\bibliography{references}

 \end{document}